\documentclass{article}%
\usepackage{amssymb}
\usepackage{amsfonts}
\usepackage{amsmath}
\usepackage{graphicx}%
\providecommand{\U}[1]{\protect\rule{.1in}{.1in}}
\newtheorem{theorem}{Theorem}

\newtheorem{example}[theorem]{Example}

\newtheorem{lemma}[theorem]{Lemma}

\newtheorem{proposition}[theorem]{Proposition}
\newtheorem{remark}[theorem]{Remark}

\newenvironment{proof}[1][Proof]{\noindent \textbf{#1.} }{\  \rule{0.5em}{0.5em}}
\begin{document}

\title{{\LARGE On the process of the eigenvalues of a Hermitian L\'{e}vy process}}
\author{Victor P\'{e}rez-Abreu\\Departamento de Probabilidad y Estad\'{\i}stica\\CIMAT, Guanajuato, M\'{e}xico\\pabreu@cimat.mx
\and Alfonso Rocha-Arteaga\\Facultad de Ciencias F\'{\i}sico-Matem\'{a}ticas\\Universidad Aut\'{o}noma de Sinaloa, M\'{e}xico\\arteaga@uas.edu.mx}
\maketitle
\date{\ \ \ \ \ \ \ \ \ \textit{Dedicated to Ole E. Barndorff-Nielsen on the
occasion of his 80th Birthday}}
\tableofcontents

\begin{abstract}
The dynamics of the eigenvalues (semimartingales) of a L\'{e}vy process $X$
with values in Hermitian matrices is described in terms of It\^{o} stochastic
differential equations with jumps. This generalizes the well known
Dyson-Brownian motion. The simultaneity of the jumps of the eigenvalues of $X$
is also studied. If $X$ has a jump at time $t$ two different situations are
considered, depending on the commutativity of $X(t)$ and $X(t-)$. In the
commutative case all the eigenvalues jump at time $t$ only when the jump of
$X$ is of full rank. In the noncommutative case, $X$ jumps at time $t$ if and
only if all the eigenvalues jump at that time when the jump of $X$ is of rank one.



\end{abstract}

\section{Introduction}

This paper is on a topic that meets a number of subjects where Ole
Barndorff-Nielsen has contributed. One is Ole's interests in the study of
infinitely divisible random matrix models and their associated matrix-valued
processes, like subordinators, L\'{e}vy and Ornstein--Uhlenbeck processes
\cite{BNPA08}-\cite{BNPAR05}, \cite{BNSe07}-\cite{BNS11}. Another is the
connection between classical and free infinite divisibility, where Ole's
papers with Thorbj\o rnsen \cite{BNT04}-\cite{BNT06} drove the study of the
so-called Upsilon transformations of classical infinitely divisible
distributions \cite{BNMS06}, \cite{BNPAT13}, \cite{BNRT08}.

Matrix L\'{e}vy processes with jumps of rank one arise as a covariation
process of two $\mathbb{C}^{d}-$valued L\'{e}vy processes \cite{DPARA} and
they provide natural examples of matrix subordinators and matrix L\'{e}vy
processes of bounded variation, such as those considered in \cite{BNPA08} and
\cite{BNSe07}.

In the context of free probability, Bercovici and Pata \cite{BP} introduced a
bijection from the set of classical infinitely divisible distributions on
$\mathbb{R}$ to the set of free infinitely divisible distributions on
$\mathbb{R}$. This bijection was described in terms of ensembles of Hermitian
random matrices $\left\{  X_{d}:d\geq1\right\}  $ by Benaych-Georges \cite{BG}
and Cabanal-Duvillard \cite{CD}. For fixed $d\geq1$, the associated $d\times
d$ matrix distribution of $X_{d}$ is invariant under unitary conjugations and
it is (matrix) infinitely divisible. The associated Hermitian L\'{e}vy process
$\left\{  X_{d}(t):t\geq0\right\}  $ has been considered in \cite{DPARA} and
it has the property that, in the non pure Gaussian case, the jumps $\Delta
X_{d}(t)=X_{d}(t)-X_{d}(t-)$ are $d\times d$ matrices of rank one, see also
\cite{DRA}, \cite{PAS}.

The simplest example of this connection is the well-known theorem of Wigner
that gives the semicircle distribution (free infinitely divisible) as the
asymptotic spectral distribution of $\left\{  X_{d}:d\geq1\right\}  $ in the
case of the Gaussian Unitary Ensemble (GUE), see \cite{AGZ09}. In this case,
for each fixed $d\geq1$, the associated L\'{e}vy process of $X_{d}$ is the
$d\times d$ Hermitian Brownian motion $\left\{  B(t):t\geq0\right\}
=\{(B_{jk}(t)):t\geq0\}$\ where $(B_{jj}(t))_{j=1}^{d},(\operatorname{Re}%
B_{jk}(t))_{j<k},$ $(\operatorname{Im}B_{jk}(t))_{j<k}$ is a set of $d^{2}$
independent one-dimensional Brownian motions with parameter $\frac{t}%
{2}(1+\delta_{jk}).$

Still in the GUE case, let $\lambda=\left\{  (\lambda_{1}(t),\lambda
_{2}(t),...,\lambda_{d}(t)):t\geq0\right\}  $ be the $d$-dimensional
stochastic process of eigenvalues of $B$. In a pioneering work, Dyson
\cite{Dy62} showed that if the eigenvalues start at different positions
($\lambda_{{\small 1}}(0)>\lambda_{{\small 2}}(0)>...>\lambda_{d}(0)$\ a.s.),
then they never meet at any time ($\lambda_{1}(t)>\lambda_{2}(t)>...>\lambda
_{d}(t)\ \ \forall$ $t>0$ a.s.) and furthermore they form a diffusion process
(and a semimartingale) satisfying the It\^{o} Stochastic Differential Equation
(SDE)
\begin{equation}
\mathrm{d}\lambda_{i}(t)=\mathrm{d}W_{i}(t)+\sum_{j\neq i}\frac{\mathrm{d}%
t}{\lambda_{i}(t)-\lambda_{j}(t)},\quad t\geq0,1\leq i\leq d, \label{1}%
\end{equation}
where $W_{1},...,W_{d}$ are independent one-dimensional standard Brownian
motions (see \cite[Sec. 4.3.1]{AGZ09}). The stochastic process $\left\{
\lambda(t):t\geq0\right\}  $ is called the Dyson non-colliding Brownian motion
corresponding to the GUE, or, in short, Dyson-Brownian motion. The study of
the eigenvalue process of other continuous matrix-valued processes has
received considerable attention, see \cite{Br89}--\cite{Ch92}, \cite{KT04}%
--\cite{AN}. The eigenvalues of all these processes are strongly dependent and
do not collide at any time almost surely.

The aim of this paper is to understand the behavior on time of the eigenvalue
process $\lambda=\{(\lambda_{1}(t),\lambda_{2}(t),...\allowbreak,\lambda
_{d}(t)):t\geq0\}$ of a $d\times d$ Hermitian L\'{e}vy process $X=\left\{
X(t):t\geq0\right\}  $. The first goal is to give conditions for the
simultaneity of the jumps of the eigenvalues. In particular, when the jump of
$X$ is of rank one and $X(t)$ and $X(t-)$ do not commute, we show that a
single eigenvalue jumps at time $t$ if and only if all the eigenvalues jump at
that time. This fact which is due to the strong dependence between the
eigenvalue processes contrasts with the case of the no simultaneity of the
jumps at any time of $d$ independent real L\'{e}vy processes; in Example
\ref{ex1} we give a Hermitian L\'{e}vy process with such eigenvalue process.
We observe that in general, the eigenvalue process is not a L\'{e}vy process
(as in the Dyson-Brownian case\textbf{)} but a semimartingale given in terms
of the Hermitian L\'{e}vy process $X.$

The second goal is to describe the dynamics of the eigenvalue process,
analogously to (\ref{1}). This leads to first considering the appropriate
It\^{o}'s formula for matrix-valued semimartingales with jumps, more general
than those of bounded variation considered by Barndorff-Nielsen and Stelzer
\cite{BNSe07}.

Our main results and the structure of the paper are as follows. In Section 2
we gather several results on Hermitian L\'{e}vy processes, including versions
of It\^{o}'s formula which are useful for studying the corresponding
eigenvalue processes. We also point out conditions under which the spectrum of
$X$ will be simple for every $t>0$ and conditions for the differentiability of
the eigenvalues. In Section 3 we consider the simultaneous jumps of the
eigenvalue (semimartingale) process of a $d\times d$ Hermitian L\'{e}vy
process $X$. When $\Delta X(t)=X(t)-X(t-)$ is not zero and $X(t)$ and $X(t-)$
commute, we show that the eigenvalues jump simultaneously if and only if the
matrix jump $\Delta X(t)$ is of rank $d.$ On the other hand, if $\Delta
X(t)\neq0$ is of rank one and $X(t)$ and $X(t-)$ do not commute, then the
process $X$ jumps at time $t>0$ if and only if all the components of $\lambda$
jump at that time. This leads to analyzing the relation between the
eigenvalues of $X(t)$ and $X(t-)$ when the jump $\Delta X(t)$ is a rank one
matrix. A key step in this direction is the recent paper \cite{RaWo11} on the
spectrum of a rank one perturbation of an unstructured matrix. Extensions to
Hermitian semimartingales are also pointed out. Finally, Section 4 considers
the dynamics of the eigenvalues of a Hermitian L\'{e}vy process, extension to
the Dyson-Brownian motion, showing that a repulsion force appears in the
bounded variation part of the eigenvalues only when there is a Gaussian
component of $X$. One of the main problems to overcome is the fact that even
when $X(t)$ is Hermitian and has a simple spectrum in an open subset almost
surely, $X(t-)$ may not.

We shall use either $X(t)$ or $X_{t}$ to denote a stochastic process in a
convenient way when the dimension or the entries of a matrix have also to be specified.

\section{Preliminaries on Hermitian L\'{e}vy processes}

In this section we gather several results on matrix-valued semimartingales and
general Hermitian L\'{e}vy processes.

Let $\mathbb{M}_{d}=\mathbb{M}_{d}\left(  \mathbb{C}\right)  $ denote the
linear space of $d\times d$ matrices with complex entries with scalar product
$\left\langle A,B\right\rangle =~\mathrm{tr}\left(  B^{\ast}A\right)  $ and
the Frobenius norm $\left\Vert A\right\Vert =\left[  \mathrm{tr}\left(
A^{\ast}A\right)  \right]  ^{1/2}$ where $\mathrm{tr}$ denotes the
(non-normalized) trace. The set of Hermitian random matrices in $\mathbb{M}%
_{d}$ is denoted by $\mathbb{H}_{d}$, $\mathbb{H}_{d}^{0}=$ $\mathbb{H}%
_{d}\backslash\{0\}$ and $\mathbb{H}_{d}^{1}$ is the set of rank one matrices
in $\mathbb{H}_{d}.$

We will use the identification of a $d\times d$ Hermitian matrix
$X=(x_{ij}+\mathrm{i}y_{ij})_{1\leq i,j\leq d}$ with an element in
$\mathbb{R}^{d^{2}}$, that is,%
\begin{equation}
X=(x_{ij}+\mathrm{i}y_{ij})_{1\leq i,j\leq d}\leftrightarrow\widehat{X}%
:=\left(  x_{11},...,x_{dd},\left(  x_{ij},y_{ij}\right)  _{1\leq i<j\leq
d}\right)  ^{\top}\text{.} \label{id}%
\end{equation}
Thus, the set $\mathbb{H}_{d}$ is identified with $\mathbb{R}^{d^{2}}$ and one
has that for $X,Y\in\mathbb{H}_{d},$ $\mathrm{tr}\left(  XY\right)
=\widehat{Y}^{\top}\widehat{X}$ is real.

An $\mathbb{M}_{d}$-valued process $X=\left\{  (x_{ij})(t):t\geq0\right\}  $
is a matrix semimartingale if $x_{ij}(t)$ is a complex semimartingale for each
$i,j=1,...,d$. Let $X=\left\{  (x_{ij})(t):t\geq0\right\}  \in\mathbb{H}_{d}$
and $Y=\left\{  (y_{ij})(t):t\geq0\right\}  \in\mathbb{H}_{d}$ be
semimartingales. Similar to the case of matrices with real entries in
\cite{BNSe07}, the matrix covariation of $X$ and $Y$ is defined as the
$\mathbb{H}_{d}$-valued process $\left[  X,Y\right]  :=\left\{  \left[
X,Y\right]  (t):t\geq0\right\}  $ with entries%
\begin{equation}
\left[  X,Y\right]  _{ij}(t)=\sum\limits_{k=1}^{d}\left[  x_{ik}%
,y_{kj}\right]  (t)\text{,} \label{DefCov}%
\end{equation}
where $\left[  x_{ik},y_{kj}\right]  (t)$ is the covariation of the
$\mathbb{C}$-valued semimartingales $\left\{  x_{ik}(t):t\geq0\right\}  $ and
$\left\{  y_{kj}(t):t\geq0\right\}  $, see \cite[pp 83]{Pr04}. One has a
decomposition into a continuous part and a pure jump part by
\begin{equation}
\left[  X,Y\right]  (t)=\left[  X^{c},Y^{c}\right]  (t)+\sum_{s\leq t}\Delta
X_{s}\Delta Y_{s}^{\ast}\text{,} \label{ForCov}%
\end{equation}
where $\left[  X^{c},Y^{c}\right]  _{ij}(t):=\sum\nolimits_{k=1}^{d}\left[
x_{ik}^{c},y_{kj}^{c}\right]  (t).$ We recall that for any semimartingale $x$,
the process $x^{c}$ is the a.s. unique continuous local martingale $m$ such
that $\left[  x-m\right]  $ is purely discontinuous.

Matrix-valued L\'{e}vy processes are examples of matrix-valued
semimartingales. In particular, so are the Hermitian L\'{e}vy processes.

\subsection{The L\'{e}vy--Khintchine representation and L\'{e}vy--It\^{o}
decomposition}

A $d\times d$ matrix Hermitian L\'{e}vy process $X=\left\{  X(t):t\geq
0\right\}  $ is characterized by the L\'{e}vy--Khintchine representation of
its Fourier transform $\mathbb{E}\mathrm{e}^{\mathrm{itr}(\Theta
X(t))}\allowbreak\ =\ \allowbreak\exp(t\psi(\Theta))$ with Laplace exponent
\[
\psi(\Theta)={}\mathrm{itr}(\Theta\Psi\text{ )}{}-{}\frac{1}{2}\mathrm{tr}%
\left(  \Theta\mathcal{A}\Theta\right)  {}+{}\int_{\mathbb{H}_{d}}\left(
\mathrm{e}^{\mathrm{itr}(\Theta\xi)}{}-1{}-\mathrm{i}\frac{\mathrm{tr}%
(\Theta\xi)}{1+\left\Vert \xi\right\Vert ^{2}}{}\right)  \nu(\mathrm{d}%
\xi),\ \Theta\in\mathbb{H}_{d},
\]
where $\mathcal{A}:\mathbb{H}_{d}\rightarrow\mathbb{H}_{d}$ is a positive
symmetric linear operator $($i.e., $\mathrm{tr}\left(  \Phi\mathcal{A}%
\Phi\right)  \geq0$ for $\Phi\in\mathbb{H}_{d}$ and $\mathrm{tr}\left(
\Theta_{2}\mathcal{A}\Theta_{1}\right)  =\mathrm{tr}\left(  \Theta
_{1}\mathcal{A}\Theta_{2}\right)  $ for $\Theta_{1},\Theta_{2}\in
\mathbb{M}_{\mathbb{H}_{d}})$, $\nu$ is a measure on $\mathbb{H}_{d}$ (the
L\'{e}vy measure) satisfying $\nu(\{0\})=0$ and $\int_{\mathbb{H}_{d}}%
(1\wedge\left\Vert x\right\Vert ^{2})\nu(\mathrm{d}x)<\infty$, and $\Psi
\in\mathbb{H}_{d}$. The triplet $(\mathcal{A},\nu,\Psi)$ uniquely determines
the distribution of $X$.

Interesting matrix Gaussian distributions for some examples of $\mathcal{A}$
are the following (here $\Psi=0$ and $\nu=0$):

(i) $\mathcal{A}=\sigma^{2}\mathrm{I}_{d}$, $\sigma^{2}>0,$ corresponds to the
GUE case of parameter $\sigma^{2}$;

(ii) More generally, $\mathcal{A}\Theta=\Sigma_{1}\Theta\Sigma_{2},$ for
$\Sigma_{1}$, $\Sigma_{2}$ nonnegative definite matrices in $\mathbb{H}_{d},$
corresponds to the matrix Gaussian distribution with Kronecker covariance
$\Sigma_{1}\otimes\Sigma_{2}$;

(iii) $\mathcal{A}\Theta=$ $\mathrm{tr}(\Theta)\sigma^{2}\mathrm{I}_{d}$,
$\sigma^{2}>0$, is the covariance operator of the Gaussian random matrix
$g\mathrm{I}_{d}$ where $g$ is a one-dimensional random variable with zero
mean and variance $\sigma^{2}$.

Any $\mathbb{H}_{d}$-valued L\'{e}vy process $X=\left\{  X(t):t\geq0\right\}
$ with triplet $(\mathcal{A},\nu,\Psi)$ is a semimartingale with the
L\'{e}vy--It\^{o} decomposition
\begin{equation}
X(t)=t\Psi+B_{\mathcal{A}}(t)+\int_{[0,t]}\int_{\left\Vert y\right\Vert \leq
1}y\widetilde{J}_{X}(\mathrm{d}s\mathrm{,d}y)+\int_{[0,t]}\int_{\left\Vert
y\right\Vert >1}yJ_{X}(\mathrm{d}s,\mathrm{d}y)\text{, }t\geq0, \label{LID}%
\end{equation}
where $\left\{  B_{\mathcal{A}}(t):t\geq0\right\}  $ is a $\mathbb{H}_{d}%
$-valued Brownian motion with covariance $\mathcal{A}$, i.e., it is a L\'{e}vy
process with continuous sample paths (a.s.) and each $B_{\mathcal{A}}(t)$ is
centered Gaussian with
\[
\mathbb{E}\left\{  \mathrm{tr(}\Theta_{1}B_{\mathcal{A}}(s){})\mathrm{tr}%
\left(  \Theta_{2}B_{\mathcal{A}}(t){}\right)  {}\right\}  =\min
(s,t)\mathrm{tr}\left(  \Theta_{1}\mathcal{A}\Theta_{2}\right)  {}\text{for
each }\Theta_{1},\Theta_{2}\in\mathbb{H}_{d};
\]
$J_{X}(\cdot,\cdot)$ is a Poisson random measure of jumps on $[0,\infty
)\times\mathbb{H}_{d}^{0}$ with intensity measure $Leb\otimes\nu$, independent
of $\left\{  B_{\mathcal{A}}(t):t\geq0\right\}  $ and the compensated measure
is given by
\[
\widetilde{J}_{X}(\mathrm{d}t,\mathrm{d}y)=J_{X}(\mathrm{d}t,\mathrm{d}%
y)-\mathrm{d}t\nu(\mathrm{d}y).
\]
For a systematic study of L\'{e}vy processes with values in Banach spaces, see
\cite{Ap07}.

\subsection{Hermitian L\'{e}vy processes with simple spectrum}

Hermitian L\'{e}vy processes with simple spectrum (all eigenvalues distinct)
are of special importance in this paper. For each $t>0$, if $X(t)$ has an
absolutely continuous distribution then $X(t)$ has a simple spectrum almost
surely, since the set of Hermitian matrices whose spectrum is not simple has
zero Lebesgue measure \cite{Tao12}, see also \cite{AGZ09}. The absolute
continuity gives also that $X(t)$ is a nonsingular random matrix since for a
Hermitian matrix $A$, $\det(A)$ is a real polynomial in $d^{2}$ variables and
$\left\{  A\in\mathbb{H}_{d}:\det(A)=0\right\}  $ has zero Lebesgue measure in
$\mathbb{R}^{d^{2}}$.

If there is a Gaussian component, i.e., $\mathcal{A\neq}{\Large 0}$, then
$X(t)$ has an absolutely continuous distribution for each $t>0$. Next we point
out a sufficient condition on the L\'{e}vy measure of a Hermitian L\'{e}vy
process without Gaussian component to have an absolutely continuous
distribution and hence a simple spectrum for each fixed time.

Let $\mathbb{S}_{d}$ be the unit sphere of $\mathbb{H}_{d}$. A polar
decomposition of the L\'{e}vy measure $\nu$ is a family of pairs $(\pi
,\rho_{\xi})$ where $\pi$, the spherical component, is a finite measure on
$\mathbb{S}_{d}$ with $\pi(\mathbb{S}_{d})\geq0$ and $\rho_{\xi}$, the radial
component, is a measure on $(0,\infty)$ for each $\xi\in\mathbb{S}_{d}$ with
$\rho_{\xi}((0,\infty))>0$ such that
\[
\nu(E)=\int_{\mathbb{S}_{d}}\pi(\mathrm{d}\xi)\int_{(0,\infty)}1_{E}(u\xi
)\rho_{\xi}(\mathrm{d}u),\qquad E\in\mathcal{B}(\mathbb{H}_{d}^{0}).
\]

We say that $\nu$ satisfies the condition $\mathbf{D}$ if, for each $\xi
\in\mathbb{S}_{d}$, there is a nonnegative function $g_{\xi}$ on $(0,\infty)$
such that $\rho_{\xi}\left(  B\right)  =\int_{B}g_{\xi}\left(  u\right)
\mathrm{d}u$ for any $B\in\mathcal{B}(\mathbb{R}\backslash\left\{  0\right\}
)$ and the following divergence condition holds%
\[
\int_{(0,\infty)}g_{\xi}\left(  u\right)  \mathrm{d}u=\infty\qquad\pi
\text{-}a.e.\text{ }\xi\text{.}%
\]

As an example, recall that a random matrix $M$ is called self-decomposable
(and hence infinitely divisible) if for any $b\in(0,1)$ there exists $M_{b}$
independent of $M$ such that $M$ and $bM+M_{b}$ have the same matrix
distribution. An infinitely divisible random matrix $M$ in $\mathbb{H}_{d}$
with triplet $(\mathcal{A},\nu,\Psi)$ is self-decomposable if and only if
\[
\nu(E)=\int_{\mathbb{S}_{d}}\pi(\mathrm{d}\xi)\int_{(0,\infty)}1_{E}%
(u\xi)\frac{k_{\xi}(u)}{u}(\mathrm{d}u),\qquad E\in\mathcal{B}(\mathbb{H}%
_{d}^{0})\text{,}%
\]
with a finite measure $\pi$ on $\mathbb{S}_{d}$ and a nonnegative decreasing
function $k_{\xi}$ on $(0,\infty)$ for each $\xi\in\mathbb{S}_{d}$. In this
case $\nu$ satisfies the condition $\mathbf{D}$ for $g_{\xi}(u)=k_{\xi}(u)/u$
where we can choose the measure $\pi$ to be zero on the set $\left\{  \xi
\in\mathbb{S}_{d}:k_{\xi}(0+)=0\right\}  $.

The following proposition for Hermitian L\'{e}vy processes follows from the
corresponding real vector case, Theorems $27.10$ and $27.13$ in \cite{Sato1},
by using the identification (\ref{id}) of $\mathbb{H}_{d}$ with $\mathbb{R}%
^{d^{2}}$. We recall that a L\'{e}vy process is nondegenerate if its
distribution for any fixed time is nondegenerate.

\begin{proposition}
\label{abscont} (a) If $\left\{  X(t):t\geq0\right\}  $ is a nondegenerate
L\'{e}vy process in $\mathbb{H}_{d}$ without Gaussian component satisfying
condition $\mathbf{D,}$ then $X(t)$ has absolutely continuous distribution for
each $t>0$.

(b) If $\left\{  X(t):t\geq0\right\}  $ is a nondegenerate self-decomposable
L\'{e}vy process in $\mathbb{H}_{d}$ without Gaussian component, then $X(t)$
has absolutely continuous distribution for each $t>0$.
\end{proposition}

\subsection{Smoothness of the spectrum of Hermitian matrices}

Next we consider several facts about eigenvalues as functions of the entries
of a Hermitian matrix, similar to \cite{AN} for symmetric matrices, that will
be useful in the sequel.

Following \cite{AGZ09}, let $\mathbb{U}_{d}^{VG}$ be the set of unitary
matrices $U=\left(  u_{ij}\right)  ,$ with $u_{ii}>0$ for all $i$, $u_{ij}%
\neq0$ for $i,j$, and all minors of $U$ having non-zero determinants. Let us
denote by $\mathbb{H}_{d}^{VG}$ the set of matrices $X$ in $\mathbb{H}_{d}$
such that
\begin{equation}
X=UDU^{\ast}, \label{decom}%
\end{equation}
where $D$ is a diagonal matrix with entries $\lambda_{i}=D_{ii}$ such that
$\lambda_{1}>\lambda_{2}>\cdots>\lambda_{d}$, $U\in\mathbb{U}_{d}^{VG}$. An
element of $\mathbb{H}_{d}^{VG}$ is called a very good matrix. The set of very
good matrices\textit{ }$\mathbb{H}_{d}^{VG}$ can be identified with an open
subset of $\mathbb{R}^{d^{2}}$. It is known that the complement of
$\mathbb{H}_{d}^{VG}$ has zero Lebesgue measure. If $T:\mathbb{U}_{d}%
^{VG}\rightarrow\mathbb{R}^{d(d-1)}$ is the mapping defined by
\[
T\left(  U\right)  =\left(  \frac{u_{12}}{u_{11}},\ldots,\frac{u_{1d}}{u_{11}%
},\frac{u_{23}}{u_{22}},\ldots,\frac{u_{2d}}{u_{22}},\ldots,\frac{u_{(d-1)d}%
}{u_{(d-1)(d-1)}}\right)  ,
\]
then Lemma $2.5.6$ in \cite{AGZ09} shows that $T$ is bijective, smooth, and
the complement of $\mathbb{U}_{d}^{VG}$ is a closed subset of zero Lebesgue measure.

Let $\mathcal{S}_{d}$ be the open set
\begin{equation}
\mathcal{S}_{d}=\left\{  \left(  \lambda_{1},\lambda_{2},\ldots,\lambda
_{d}\right)  \in\mathbb{R}^{d}:\lambda_{1}>\lambda_{2}>\cdots>\lambda
_{d}\right\}  \label{SS}%
\end{equation}
and let $D_{\lambda}$ be the diagonal matrix $D_{ii}=\lambda_{i}$ for any
$\lambda=\left(  \lambda_{1},\lambda_{2},\ldots,\lambda_{d}\right)
\in\mathcal{S}_{d}$. Then the map $\widetilde{T}:\mathcal{S}_{d}\times
T\left(  \mathbb{U}_{d}^{VG}\right)  \rightarrow\mathbb{H}_{d}^{VG}$ defined
by $\widetilde{T}\left(  \mathbb{\lambda},z\right)  =T^{-1}(z)^{\ast
}D_{\lambda}T^{-1}(z)$ is a smooth bijection and the inverse mapping of
$\widetilde{T}$ satisfies%
\[
\widetilde{T}^{-1}\left(  X\right)  =\left(  \lambda\left(  X\right)
,\widetilde{U}(X)\right)  \text{,}%
\]
where $\widetilde{U}(X)=T(U)$ for $U\in\mathbb{U}_{d}^{VG}$ satisfying
(\ref{decom}). As a consequence of these results, there exist an open subset
$\widetilde{G}\subset\mathbb{R}^{d^{2}}$ such that $\mathbb{R}^{d^{2}%
}\backslash\widetilde{G}$ has zero Lebesgue measure and a smooth function%
\begin{equation}
\Phi:\mathbb{R}^{d^{2}}\rightarrow\mathbb{R}^{d} \label{fi}%
\end{equation}
such that $\lambda\left(  X\right)  =\Phi\left(  X\right)  $ for
$X\in\widetilde{G}$.

\subsection{The It\^{o} Formula for Hermitian L\'{e}vy processes}

We shall first consider a convenient notation for the Fr\'{e}chet derivative
of a function $f$ from $\mathbb{H}_{d}$ to $\mathbb{R}$ (see \cite[X.4]%
{Bh97}), using the identification (\ref{id}) of an element $X=(x_{ij}%
+\mathrm{i}y_{ij})_{1\leq i,j\leq d}$ in $\mathbb{H}_{d}$ as the element
$\left(  x_{11},...,x_{dd},\left(  x_{ij},y_{ij}\right)  _{1\leq i<j\leq
d}\right)  ^{\top}$ in $\mathbb{R}^{d^{2}}$.

The derivative of $f$ at the point $u$ is denoted by $Df(u)$; it can be
thought as a linear operator from $\mathbb{R}^{d^{2}}$ to $\mathbb{R}$. Its
representation as a Hermitian matrix is%
\begin{equation}
Df=\left(  \frac{\partial}{\partial x_{ij}}f+\mathrm{i}\frac{\partial
}{\partial y_{ij}}f1_{\left\{  i\neq j\right\}  }\right)  _{1\leq i\leq j\leq
d}\text{.} \label{DM}%
\end{equation}
Let $L\left(  \mathbb{R}^{d^{2}},\mathbb{R}\right)  $ denote the space of
linear operators from $\mathbb{R}^{d^{2}}$ to $\mathbb{R}$. If the map
$Df:\mathbb{R}^{d^{2}}\rightarrow L\left(  \mathbb{R}^{d^{2}},\mathbb{R}%
\right)  $ defined by $u\rightarrow Df(u)$ is differentiable at a point $u$,
we say that $f$ is twice differentiable at $u$. We denote the derivative of
the map $Df$ at the point $u$ by $D^{2}f(u)$. Recall that $D^{2}f(u)$ is a
bilinear operator from $\mathbb{R}^{d^{2}}\times\mathbb{R}^{d^{2}}$ to
$\mathbb{R}$ which is symmetric, that is, $D^{2}f(u)(x,y)=D^{2}f(u)(y,x)$. We
use the following notation for the first and second partial derivative
operators, where $z_{ij}=x_{ij}$ or $z_{ij}=y_{ij}$, with $y_{ii}=0,$
$i=1,...,d$%
\begin{equation}
\sum\limits_{\left(  i,j\right)  }\frac{\partial}{\partial z_{ij}}%
:=\sum\limits_{1\leq i\leq d}\frac{\partial}{\partial x_{ii}}+2\sum
\limits_{1\leq i<j\leq d}(\frac{\partial}{\partial x_{ij}}+\frac{\partial
}{\partial y_{ij}}), \label{OFD}%
\end{equation}%
\begin{align}
\sum\limits_{\left(  i,j\right)  \left(  k,h\right)  }\frac{\partial^{2}%
}{\partial z_{ij}z_{kh}}  &  :=\sum\limits_{1\leq i\leq j\leq d}\frac
{\partial^{2}}{\partial x_{ii}x_{jj}}+\sum\limits_{1\leq i\leq j\leq d}%
\sum\limits_{1\leq k<h\leq d}(\frac{\partial^{2}}{\partial x_{ij}x_{kh}}%
+\frac{\partial^{2}}{\partial x_{ij}y_{kh}})\nonumber\\
&  +\sum\limits_{1\leq i<j\leq d}\sum\limits_{1\leq k<h\leq d}\frac
{\partial^{2}}{\partial y_{ij}y_{kh}}. \label{OSD}%
\end{align}
For a linear operator $\mathcal{A}:\mathbb{H}_{d}\rightarrow\mathbb{H}_{d}$,
we denote its $d^{2}\times d^{2}$ matrix representation by $\mathcal{A=(}%
A_{ij,kh})$.

Following Barndorff-Nielsen and Stelzer \cite{BNSe07}, we say that for each
open subset $G\subset\mathbb{M}_{d}$, the process $X=\left\{  X_{t}%
:t\geq0\right\}  $ is locally bounded within $G$ if there exist a sequence of
stopping times $\left(  T_{n}\right)  _{n\geq1}$ increasing to infinity almost
surely and a sequence of compact convex subsets $D_{n}\subset G$ with
$D_{n}\subset D_{n+1}$ such that $X_{t}\in D_{n}$ for all $0\leq t<T_{n}$.
Observe that if the process $X$ is locally bounded within $G$, it cannot get
arbitrarily close to the boundary of $G$ in finite time, hence $X_{t}$ and
$X_{t-}$ are in $D_{n}$ at all times $t\geq0$.

The following result can be proved similarly to Proposition $3.10.10$ in
Bichteler \cite{Bi01}, who considered the multivariate case. The bounded
variation multivariate case can be seen in \cite{BNSe07}.

\begin{theorem}
[The It\^{o} formula for matrix semimartingales]\label{ito1} Let $\left\{
X_{t}:t\geq0\right\}  $ be a matrix semimartingale, $G\subset\mathbb{H}_{d}$
an open subset, and $f:G\rightarrow\mathbb{R}$ continuously differentiable. If
$\left\{  X_{t}:t\geq0\right\}  $ is locally bounded within $G$, then $X_{t-}$
belongs to $G$ for all $t>0,$ the process $f\left(  X_{t}\right)  $ is a
semimartingale, and%
\[
f\left(  X_{t}\right)  =f\left(  X_{0}\right)  +\mathrm{tr}\left(
\int\nolimits_{0+}^{t}Df\left(  X_{s-}\right)  dX_{s}\right)  +\frac{1}{2}%
\int\nolimits_{0+}^{t}\sum\limits_{\left(  i,j\right)  \left(  k,h\right)
}\frac{\partial^{2}}{\partial z_{ij}z_{kh}}f\left(  X_{s}\right)  d[z_{ij}%
^{c},z_{kh}^{c}]_{s}%
\]%
\[
+\int\nolimits_{\left(  0,t\right]  \times\mathbb{H}_{d}^{0}}\left[  f\left(
X_{s}\right)  -f\left(  X_{s-}\right)  -\mathrm{tr}\left(  Df\left(
X_{s-}\right)  y\right)  \right]  J_{X}\left(  ds,dy\right)  .
\]

\end{theorem}

The following two propositions for Hermitian L\'{e}vy processes are special
cases of the above result. The first one is the analog of the multivariate
case in \cite[Prop. 8.15]{CT}.

\begin{proposition}
[The It\^{o} formula for matrix L\'{e}vy processes]\label{ito2} Let $\left\{
X_{t}:t\geq0\right\}  $ be a Hermitian L\'{e}vy process with triplet $\left(
\mathcal{A},\nu,\Psi\right)  $, $G\subset\mathbb{H}_{d}$ an open subset, and
$f:G\rightarrow\mathbb{R}$ twice continuously differentiable. If $\left\{
X_{t}:t\geq0\right\}  $ is locally bounded within $G$, then $X_{t-}$ belongs
to $G$ for all $t>0,$ the process $f\left(  X_{t}\right)  $ is a
semimartingale, and%
\begin{equation}
f\left(  X_{t}\right)  =f\left(  X_{0}\right)  +\mathrm{tr}\left(
\int\nolimits_{0}^{t}Df\left(  X_{s-}\right)  dX_{s}\right)  +\frac{1}{2}%
\int\nolimits_{0}^{t}\sum\limits_{\left(  i,j\right)  \left(  k,h\right)
}A_{ij,kh}\frac{\partial^{2}}{\partial z_{ij}z_{kh}}f\left(  X_{s}\right)  ds
\label{iflp}%
\end{equation}%
\[
+\int\nolimits_{\left(  0,t\right]  \times\mathbb{H}_{d}^{0}}\left[  f\left(
X_{s-}+y\right)  -f\left(  X_{s-}\right)  -\mathrm{tr}\left(  Df\left(
X_{s-}\right)  y\right)  \right]  J_{X}\left(  ds,dy\right)  .
\]

\end{proposition}

The next martingale-drift decomposition of functions of a Hermitian L\'{e}vy
proces $X$ can be obtained by replacing the L\'{e}vy-It\^{o} decomposition
(\ref{LID}) for $X$ into the stochastic integral part of (\ref{iflp}).

\begin{proposition}
\label{ito3} Let $\left\{  X_{t}:t\geq0\right\}  $ be a Hermitian L\'{e}vy
process with triplet $\left(  \mathcal{A},\nu,\Psi\right)  $, $G\subset
\mathbb{H}_{d}$ an open subset, and $f:G\rightarrow\mathbb{R}$ twice
continuously differentiable. If $\left\{  X_{t}:t\geq0\right\}  $ is locally
bounded within $G$, then $X_{t-}$ belongs to $G$ for all $t>0,$ and the
process $f\left(  X_{t}\right)  $ can be decomposed as $f\left(  X_{t}\right)
=M_{t}+V_{t}$ where $M_{t}$ is a martingale and $V_{t}$ is a bounded variation
process such that%
\[
M_{t}=f\left(  X_{0}\right)  +\mathrm{tr}\left(  \int\nolimits_{0}%
^{t}Df\left(  X_{s-}\right)  B_{\mathcal{A}}(ds)\right)  +\int%
\nolimits_{\left(  0,t\right]  \times\mathbb{H}_{d}^{0}}\left[  f\left(
X_{s-}+y\right)  -f\left(  X_{s-}\right)  \right]  \widetilde{J}_{X}\left(
ds,dy\right)
\]
and%
\[
V_{t}=\frac{1}{2}\int\nolimits_{0}^{t}\sum\limits_{\left(  i,j\right)  \left(
k,h\right)  }A_{ij,kh}\frac{\partial^{2}}{\partial z_{ij}z_{kh}}f\left(
X_{s}\right)  ds+\mathrm{tr}\left(  \int\nolimits_{0}^{t}\Psi Df\left(
X_{s}\right)  ds\right)  ~~~~~~~~~~~~~~~~~~~~~~~~~~~
\]%
\[
~~~~~~~~~~~~~~~~~~~~~~~~~+\int\nolimits_{\left(  0,t\right]  \times
\mathbb{H}_{d}^{0}}\left[  f\left(  X_{s-}+y\right)  -f\left(  X_{s-}\right)
-\mathrm{tr}\left(  Df\left(  X_{s-}\right)  y\right)  1_{\{ \left\Vert
y\right\Vert \leq1\}}\right]  ds\nu(dy).
\]

\end{proposition}

\begin{remark}
\label{ito4} If a Hermitian L\'{e}vy process $X$ is invariant under unitary
conjugation with triplet $\left(  \sigma^{2}\mathrm{I}_{d^{2}},\nu
,\psi\mathrm{I}_{d}\right)  $, where $\nu$ has the polar decomposition
$(\pi,\rho)$, and $\pi$ is the uniform measure on $\mathbb{S}_{d}$, then the
It\^{o} formula for $X$ in Proposition \ref{ito3} becomes in
\begin{align*}
M_{t}  &  =f\left(  X_{0}\right)  +\mathrm{tr}\left(  \int\nolimits_{0}%
^{t}Df\left(  X_{s-}\right)  B_{\mathcal{A}}(ds)\right) \\
&  ~~~~~~~~~~~~\int\nolimits_{\left(  0,t\right]  \times\mathbb{S}_{d}%
\times(0,\infty)}\left[  f\left(  X_{s-}+ru\right)  -f\left(  X_{s-}\right)
\right]  \left(  J_{X}\left(  ds,du,dr\right)  -ds\pi(du)\rho(dr)\right)
\end{align*}
and%
\begin{align*}
V_{t}  &  =\frac{\sigma^{2}}{2}\int\nolimits_{0}^{t}\sum\limits_{\left(
i,j\right)  \left(  k,h\right)  }\frac{\partial^{2}}{\partial z_{ij}z_{kh}%
}f\left(  X_{s}\right)  ds+\psi\mathrm{tr}\left(  \int\nolimits_{0}%
^{t}Df\left(  X_{s}\right)  ds\right) \\
&  +\int\nolimits_{\left(  0,t\right]  \times\mathbb{S}_{d}\times(0,\infty
)}\left[  f\left(  X_{s-}+ru\right)  -f\left(  X_{s-}\right)  -r\mathrm{tr}%
\left(  Df\left(  X_{s-}\right)  u\right)  1_{\{r\leq1\}}\right]
ds\pi(du)\rho(dr).
\end{align*}

\end{remark}

\section{Simultaneous jumps of a Hermitian L\'{e}vy process and its
eigenvalues}

Let $\left\{  X(t):t\geq0\right\}  $ be a $d\times d$ Hermitian L\'{e}vy
process and let $\lambda(t)=(\lambda_{1}(t),...,\lambda_{d}(t))$ be the vector
of eigenvalues of $X(t)$ where $\lambda_{1}(t)\geq\lambda_{2}(t)\geq\cdots
\geq\lambda_{d}(t)$ for each $t\geq0$. If $\lambda$ jumps at time $t,$ it is
clear that $\Delta X(t)\neq0$, since the $d$-tuple of ordered eigenvalues is a
continuous function on the space of Hermitian matrices. This follows from the
Hoffman-Wielandt inequality (see \cite[Rem. $2.1.20$]{AGZ09}) where for each
$t\geq0$%
\begin{equation}
\left\Vert \Delta\lambda(t)\right\Vert _{2}\leq\left\Vert \Delta
X(t)\right\Vert , \label{WH}%
\end{equation}
where $\left\Vert \text{\quad}\right\Vert _{2}$ denotes the Euclidean norm in
$\mathbb{R}^{d}$.

A natural question is whether the jump of the process $X$ at time $t$ implies
the jump of the eigenvalues at that time. Our goal in this section is to
provide conditions for the simultaneous jumps of the eigenvalue process
$\lambda$ whenever $X$ jumps.

We analyze the non zero jumps of $X$ in two situations, when $X(t)$ and
$X(t-)$ commute and when they do not commute. In the commutative case we find
that exactly $k$ eigenvalues jump when $X$ has a jump of rank $k$ at time $t$,
and that all the eigenvalues would jump simultaneously if and only if the jump
of $X$ at time $t$ is of rank $d.$ In the noncommutative case, we prove that
all the eigenvalues jump simultaneously if the jump of $X$ at time $t$ is of
rank one.

We present below natural examples of both, commutative and noncommutative situations.

\begin{example}
\label{ex1}Let $\left\{  (\lambda_{1}(t),...,\lambda_{d}(t)):t\geq0\right\}  $
be a $d$-dimensional L\'{e}vy process where $\{ \lambda_{i}(t):t\geq0\}$,
$i=1,...,d,$ are independent one-dimensional L\'{e}vy processes. For an
arbitrary non random unitary $d\times d$ matrix $U$, define the $d\times d$
Hermitian L\'{e}vy process
\[
X(t)=Udiag(\lambda_{1}(t),...,\lambda_{d}(t))U^{\ast},\quad t\geq0.
\]
Since independent L\'{e}vy processes cannot jump at the same time almost
surely, $X$ is an Hermitian L\'{e}vy process whose jumps are of rank one
almost surely and we observe that in that case $X(t)$ and $X(t-)$ are
simultaneously diagonalizable.
\end{example}

\begin{example}
\label{ex2}Let us consider a compound Poisson $d\times d$ matrix process%
\[
X(t)=\sum\limits_{j=1}^{N(t)}u_{j}u_{j}^{\ast},\quad t\geq0,
\]
where $u_{j},j\geq1$ are i.i.d. random vectors uniformly distributed on the
unit sphere of $\mathbb{C}^{d}$ and $N(t),t\geq0,$ is a Poisson process
independent of $u_{j},j\geq1$. Since $\bigtriangleup X(t)=u_{N(t)}%
u_{N(t)}^{\ast}1_{\left\{  \bigtriangleup X(t)=1\right\}  }$ the process $X$
has jumps of rank one almost surely and $X(t)=X(t-)+u_{N(t)}u_{N(t)}^{\ast}$
whenever there is a jump at time $t$. Hence we observe that $X(t)$ and $X(t-)$
do not commute since $u_{N(t)}u_{N(t)}^{\ast}$ and $X(t-)=u_{1}u_{1}^{\ast
}+\cdots+u_{N(t)-1}u_{N(t)-1}^{\ast}$ do not commute.
\end{example}

\begin{example}
\label{ex3}Let us consider $\left\{  X(t):t\geq0\right\}  $ a Hermitian
L\'{e}vy process in $\mathbb{H}_{d}^{+}$ the set\ of nonnegative definite
$d\times d$ Hermitian matrices. Assume that $X$ has jumps of rank one. Then by
Theorem $4.1$ in \cite{DPARA} $X$ is the quadratic variation of a
$\mathbb{C}^{d}$-valued L\'{e}vy process $\left\{  Y(t):t\geq0\right\}  $,
that is
\[
X(t)=\left[  Y\right]  (t),\quad t\geq0\text{.}%
\]
For each $t>0$, $X(t)=\sum\limits_{0<u\leq t}\bigtriangleup Y_{u}\left(
\bigtriangleup Y_{u}\right)  ^{\ast}$ and $X(t-)=\lim\limits_{s\nearrow t}%
\sum\limits_{0<r\leq s}\bigtriangleup Y_{r}\left(  \bigtriangleup
Y_{r}\right)  ^{\ast}$ do not commute in general. In particular, $X(t)$ and
$X(t-)$ commute if for each $u,r\in(0,t]$ the term $\bigtriangleup
Y_{u}\left(  \bigtriangleup Y_{u}\right)  ^{\ast}\bigtriangleup Y_{r}\left(
\bigtriangleup Y_{r}\right)  ^{\ast}$ is Hermitian.
\end{example}

\begin{example}
\label{ex4}Let $\left\{  Z(t):t\geq0\right\}  $ be a Hermitian L\'{e}vy
process in $\mathbb{H}_{d}$ of bounded variation with jumps of rank one. By
Theorem $4.2$ in \cite{DPARA} $Z$ is the difference of the quadratic
variations of two $\mathbb{C}^{d}$-valued L\'{e}vy processes $\left\{
X(t):t\geq0\right\}  $ and $\left\{  Y(t):t\geq0\right\}  $,%
\[
Z(t)=\left[  X\right]  (t)-\left[  Y\right]  (t),\quad t\geq0,
\]
where $\left[  X\right]  (t),t\geq0\ $and$\ \left[  Y\right]  (t),t\geq0$ are
independent L\'{e}vy processes in $\mathbb{H}_{d}^{+}$. It follows from
Example \ref{ex3} that $Z(t)=\left[  X\right]  (t)-\left[  Y\right]  (t)$ and
$Z(t-)=\left[  X\right]  (t-)-\left[  Y\right]  (t-)$ do not commute in
general. In particular, $Z(t)$ and $Z(t-)$ commute if for each $u,r\in(0,t]$
the terms $\bigtriangleup X_{u}\left(  \bigtriangleup X_{u}\right)  ^{\ast
}\bigtriangleup X_{r}\left(  X_{r}\right)  ^{\ast},$ $\bigtriangleup
Y_{u}\left(  \bigtriangleup Y_{u}\right)  ^{\ast}\bigtriangleup Y_{r}\left(
\bigtriangleup Y_{r}\right)  ^{\ast}$ and $\bigtriangleup X_{u}\left(
\bigtriangleup X_{u}\right)  ^{\ast}\bigtriangleup Y_{r}\left(  \bigtriangleup
Y_{r}\right)  ^{\ast}$ are Hermitian.
\end{example}

\subsection{Commutative case}

Assume that $X(t)$ and $X(t-)$ commute. Then there exists a random unitary
matrix $U=U_{t}$ such that $X(t)=Udiag(\lambda_{1}(t),\allowbreak
...,\allowbreak\lambda_{d}(t))\allowbreak U^{\ast}$ and $X(t-)=U\allowbreak
diag\allowbreak(\lambda_{1}(t-),\allowbreak...,\allowbreak\lambda
_{d}(t-))\allowbreak U^{\ast}$. Observe that the eigenvalues of
$\bigtriangleup X(t)=X(t)-X(t-)$ are the difference between the eigenvalues of
$X(t)$ and the eigenvalues of $X(t-)$.

If $\bigtriangleup X(t)$ is of rank one, there is exactly one eigenvalue
$\lambda_{i}$ such that $\bigtriangleup X(t)=Udiag(0,...,\bigtriangleup
\lambda_{i}(t),...,0)U^{\ast}$ and all the other eigenvalues do not jump at
that time $t$. Similarly if $\bigtriangleup X(t)$ is of rank $k$ where $1\leq
k\leq d$ there are exactly $k$ eigenvalues that jump at time $t$ and all other
$d-k$ eigenvalues remain without jumping at $t$. Observe that the only case
when all eigenvalues jump simultaneously at $t$ is when the jump
$\bigtriangleup X(t)\neq0$ is of rank $d$.

\subsection{Noncommutative case}

When $X(t)$ and $X(t-)$ do not commute the situation is completely different
from the commutative case, since the eigenvalues of $X(t)-X(t-)$ are not any
more the difference between the eigenvalues of $X(t)$ and the eigenvalues of
$X(t-)$. In this case we prove that if the jump $\bigtriangleup X(t)$ is of
rank one, all eigenvalues jump simultaneously at this time $t$, as we will see
in Theorem \ref{sim}.

We first consider some facts about the eigenvalues of rank one matrix
perturbations. Let $A$ be a $d\times d$ complex matrix. A rank one
perturbation of $A$ is a $d\times d$ complex matrix $A+ruv^{\top}$ where
$u,v\in\mathbb{C}^{d}$ and $r\in\mathbb{C}$. If $S$ is any set of polynomials
in $\mathbb{C}[x_{1},...,x_{d}]$, we define $V(S)=\left\{  u\in\mathbb{C}%
^{d}:f(u)=0\text{ for all }f\in S\right\}  $. If $S$ is a finite set of
polynomials $f_{1},...,f_{n}$ we write $V(f_{1},...,f_{n})$ for $V(S)$. Recall
that a subset $W$ of $\mathbb{C}^{d}$ is an algebraic set if $W=V(f_{1}%
,...,f_{n})$ for some $f_{1},...,f_{n}$. A non empty subset $Q\subset
\mathbb{C}^{d}$ is generic if the complement $\mathbb{C}^{d}\backslash Q\ $is
contained in an algebraic set $W$ which is not $\mathbb{C}^{d}$. In such a
case, $\mathbb{C}^{d}\backslash Q$ is nowhere dense and of $2d$ dimensional
Lebesgue measure zero.

When a Hermitian L\'{e}vy process $X$ has simple spectrum for each $t>0$, then
$X_{t}$ and $X_{t-}$ have no common eigenvalues when the jump at time $t$ is
of rank one.

\begin{proposition}
\label{eig}Let $\{X_{t}:t\geq0\}$ be an absolutely continuous $d\times d$
Hermitian L\'{e}vy process. Let $\lambda(t)=(\lambda_{1}(t),...,\lambda
_{d}(t))$ be the vector of eigenvalues of $X_{t}$ for each $t\geq0$. If $X$
has a rank one jump at time $t>0$, then%
\[
\left\{  \lambda_{1}(X_{t}),\ldots,\lambda_{d}(X_{t})\right\}  \cap\left\{
\lambda_{1}(X_{t-}),\ldots,\lambda_{d}(X_{t-})\right\}  =\varnothing\text{
a.s.}%
\]

\end{proposition}

\begin{proof}
Observe that $Q=\left\{  z\in\mathbb{C}^{2d}:tr\left(  zz^{\ast}\right)
\neq0\right\}  $ is a generic set since $\{z\in\mathbb{C}^{2d}:\allowbreak
tr\left(  zz^{\ast}\right)  =0\}=\left\{  0\right\}  $ is an algebraic set
which is equal to $V(f_{1},...,f_{2d})$ where $f_{i}(z)=z_{i}$ for any
$z=\left(  z_{1},...,z_{2d}\right)  ^{\top}\in\mathbb{C}^{d}$ and $i=1,...,2d$.

Since $X$ has a jump of rank one at time $t$, $\bigtriangleup X_{t}=ruu^{\ast
}$ with $r=r_{t}\in\mathbb{R}\backslash\{0\}$ and $u=u_{t}\in\mathbb{C}%
^{d}\backslash\{0\}.$ That is, we have the rank one perturbation $X_{t-}%
=X_{t}+ruu^{\ast}$. Since $X_{t}$ is absolutely continuous and Hermitian, the
eigenvalues $\lambda_{1}(X_{t}),\cdots,\lambda_{d}(X_{t})$ do not collide
almost surely, hence the degree of the minimal polynomial of $X_{t}$ is $d$.
By Proposition $2.2$ (iii) of \cite{RaWo11} applied to $z=(u,\overline{u})\in
Q$, there are exactly $d$ eigenvalues of $X_{t-}$ that are not eigenvalues of
$X_{t}$.
\end{proof}

\begin{remark}
Proposition $2.2$ in \cite{RaWo11} also gives insight into the behavior of the
non-Hermitian rank one jumps of a general complex matrix L\'{e}vy process for
which $X_{t}$ no longer has a simple spectrum. In this case, $\bigtriangleup
X_{t}=ruv^{\top}$ with $r\in\mathbb{C}\backslash\{0\}$ and $u,v\in
\mathbb{C}^{d}\backslash\{0\},$ the number $m_{t}\leq d$ of distinct
eigenvalues of $X_{t}$ might be random and there are exactly $m_{t}$
eigenvalues of $X_{t-}$ that are not eigenvalues of $X_{t}$.
\end{remark}

We now present the main result of this paper, giving conditions for the
simultaneous jumps of a Hermitian L\'{e}vy process and its eigenvalue process.
Moreover, if one eigenvalue component of the eigenvalue process jumps, then
all the other eigenvalue components jump.

\begin{theorem}
\label{sim}Let $\left\{  X(t):t\geq0\right\}  $ be an absolutely continuous
$d\times d$ Hermitian L\'{e}vy process and let $\lambda(t)=(\lambda
_{1}(t),...,\lambda_{d}(t))$ be the vector of eigenvalues of $X(t)$ where
$\lambda_{1}(t)\geq\lambda_{2}(t)\geq\cdots\geq\lambda_{d}(t)$ for each
$t\geq0$.

If $X$ has a jump of rank one at $t>0$ and $X(t)$ and $X(t-)$ do not commute
then $\Delta\lambda(t)\neq0$. Moreover $\Delta\lambda_{i}\left(  t\right)
\neq0$ for some $i$ if and only if $\Delta\lambda_{j}\left(  t\right)  \neq0$
for all $j\neq i.$
\end{theorem}

\begin{proof}
If the L\'{e}vy process $X$ jumps at time $t$, then, by Proposition \ref{eig},
$\lambda_{i}(X_{t})-\lambda_{j}(X_{t-})\neq0$ for all $i,j=1,2,...,d$ and
therefore $\Delta\lambda(t)\neq0.$

Moreover, if $\lambda$ jumps at time $t$, it is clear from (\ref{WH}) that
$\Delta X(t)\neq0$, and then by the first part of the proof all the
eigenvalues jump, that is $\Delta\lambda_{j}\left(  t\right)  \neq0$ for all
$j\neq i.$
\end{proof}

\begin{remark}
a) We finally observe that Proposition \ref{eig} and Theorem \ref{sim} are
true also for a Hermitian semimartingale $X$ with simple spectrum for each
$t>0$. This is guaranteed whenever $X(t)$ has an absolutely continuous
distribution for each $t>0.$

b) However, we have written the results of this section in the framework of
Hermitian L\'{e}vy processes since the subject is one of Ole's interests.
\end{remark}

\section{The dynamics of the eigenvalues of Hermitian L\'{e}vy processes}

\subsection{Eigenvalue derivatives of Hermitian L\'{e}vy processes}

Consider a Hermitian matrix $X$ with decomposition $X=UDU^{\ast}$ given by
(\ref{decom})\ where $U=\left(  u_{ij}\right)  $ is a\ unitary matrix$,$ with
$u_{ii}>0$ for all $i$, $u_{ij}\neq0$ for $i,j$, and all minors of $U$ having
non-zero determinants and $D$ is a diagonal matrix with entries $\lambda
_{i}=D_{ii}$ such that $\lambda_{1}>\lambda_{2}>\cdots>\lambda_{d}$.

Now if we assume that $X=(x_{ij}+\mathrm{i}y_{ij})_{1\leq i,j\leq d}$ is a
smooth function of a real parameter $\theta$, the Hadamard variation formulae
hold: for each $m=1,2,...,d,$%
\[
\partial_{\theta}\lambda_{m}=\left(  U^{\ast}\partial_{\theta}XU\right)
_{mm},
\]%
\[
\partial_{\theta}^{2}\lambda_{m}=\left(  U^{\ast}\partial_{\theta}%
^{2}XU\right)  _{mm}+2\sum\limits_{j\neq m}\frac{\left\vert \left(  U^{\ast
}\partial_{\theta}XU\right)  _{mj}\right\vert ^{2}}{\lambda_{m}-\lambda_{j}%
}\text{.}%
\]
In particular, for $\theta=x_{kh}$ and $y_{kh}$ with $k\leq h$, we observe
that%
\[
\frac{\partial\lambda_{m}}{\partial x_{kh}}=2\operatorname{Re}(\overline
{u}_{km}u_{hm})1_{\left\{  k\neq h\right\}  }+2\left\vert u_{km}\right\vert
^{2}1_{\left\{  k=h\right\}  ,}%
\]

\[
\frac{\partial\lambda_{m}}{\partial y_{kh}}=2\operatorname{Im}(\overline
{u}_{km}u_{hm})1_{\left\{  k\neq h\right\}  },
\]
\medskip%
\[
\frac{\partial^{2}\lambda_{m}}{\partial x_{kh}^{2}}=2\sum\limits_{j\neq
m}\frac{\left\vert \overline{u}_{km}u_{hj}+\overline{u}_{hm}u_{kj}\right\vert
^{2}}{\lambda_{m}-\lambda_{j}}1_{\left\{  k\neq h\right\}  }+2\sum
\limits_{j\neq m}\frac{\left\vert \overline{u}_{km}u_{kj}\right\vert ^{2}%
}{\lambda_{m}-\lambda_{j}}1_{\left\{  k=h\right\}  ,}%
\]

\[
\frac{\partial^{2}\lambda_{m}}{\partial y_{kh}^{2}}=2\sum\limits_{j\neq
m}\frac{\left\vert \overline{u}_{km}u_{hj}-\overline{u}_{hm}u_{kj}\right\vert
^{2}}{\lambda_{m}-\lambda_{j}}1_{\left\{  k\neq h\right\}  \text{.}}%
\]

We point out that the spectral dynamics of Hermitian matrices is studied in
\cite{Tao11}, giving a Hadamard variation formula of any order for each
eigenvalue in terms of a gradient. It is noted there that an eigenvalue
depends smoothly on $X$ whenever that eigenvalue is simple (even if other
eigenvalues are not).

Using the former derivatives of the eigenvalues and the smooth function given
in (\ref{fi}) we have the following result for Hermitian L\'{e}vy processes
with simple spectrum for each time $t>0.$

\begin{lemma}
\label{Der} Let $\left\{  X_{t}:t\geq0\right\}  $ be a $d\times d$ Hermitian
L\'{e}vy process with absolutely continuous distribution for each $t>0$. Let
$\lambda(t)=(\lambda_{1}(t),...,\lambda_{d}(t))$ with $\lambda_{1}%
(t)>\cdots>\lambda_{d}(t)$ the eigenvalues of $X_{t}$ for each $t\geq0$. For
each $m=1,2,...,d$, there exists $\Phi_{m}:\mathbb{R}^{d^{2}}\rightarrow
\mathbb{R}$ which is $C^{\infty}$ in an open subset $\widetilde{G}%
\subset\mathbb{R}^{d^{2}},$ with $Leb\left(  \mathbb{R}^{d^{2}}\backslash
\widetilde{G}\right)  =0$, such that for each $t>0$, $\lambda_{m}(t)=\Phi
_{m}\left(  X_{t}\right)  $ and for $k\leq h$%
\begin{equation}
\frac{\partial\Phi_{m}}{\partial x_{kh}}=2\operatorname{Re}(\overline{u}%
_{km}u_{hm})1_{\left\{  k\neq h\right\}  }+2\left\vert u_{km}\right\vert
^{2}1_{\left\{  k=h\right\}  ,} \label{fix}%
\end{equation}%
\begin{equation}
\frac{\partial\Phi_{m}}{\partial y_{kh}}=2\operatorname{Im}(\overline{u}%
_{km}u_{hm})1_{\left\{  k\neq h\right\}  }, \label{fiy}%
\end{equation}
and%
\begin{equation}
\frac{\partial^{2}\Phi_{m}}{\partial x_{kh}^{2}}=2\sum\limits_{j\neq m}%
\frac{\left\vert \overline{u}_{km}u_{hj}+\overline{u}_{hm}u_{kj}\right\vert
^{2}}{\lambda_{m}-\lambda_{j}}1_{\left\{  k\neq h\right\}  }+2\sum
\limits_{j\neq m}\frac{\left\vert \overline{u}_{km}u_{kj}\right\vert ^{2}%
}{\lambda_{m}-\lambda_{j}}1_{\left\{  k=h\right\}  ,} \label{fix2}%
\end{equation}%
\begin{equation}
\frac{\partial^{2}\Phi_{m}}{\partial y_{kh}^{2}}=2\sum\limits_{j\neq m}%
\frac{\left\vert \overline{u}_{km}u_{hj}-\overline{u}_{hm}u_{kj}\right\vert
^{2}}{\lambda_{m}-\lambda_{j}}1_{\left\{  k\neq h\right\}  }, \label{fiy2}%
\end{equation}
where for simplicity, the dependence on $t$ is omitted in $\lambda_{m}%
=\lambda_{m}(t)$ and $u_{kl}=u_{kl}(t),$ which are random.
\end{lemma}

\subsection{The SDE for the process of eigenvalues}

Using Lemma \ref{Der} and Proposition \ref{ito2}, we can describe the dynamics
of the process of eigenvalues of a Hermitian L\'{e}vy process.

\begin{theorem}
\label{Dynamics} Let $\left\{  X_{t}:t\geq0\right\}  $ be a Hermitian L\'{e}vy
process with triplet $\left(  \mathcal{A},\nu,\Psi\right)  $ and locally
bounded within the open set $G$ in $\mathbb{H}_{d}$ identified with the open
set $\widetilde{G}$ in $\mathbb{R}^{d^{2}}$ of Lemma \ref{Der}. Assume that
$X_{t}$ has an absolutely continuous distribution for each $t>0$. Let
$\lambda(t)=(\lambda_{1}(t),...,\lambda_{d}(t))$ be the eigenvalues of $X_{t}$
for each $t\geq0.$

Then, for each $m=1,\ldots,d$ the process $\lambda_{m}\left(  X_{t}\right)  $
is a semimartingale and%
\begin{align}
\lambda_{m}\left(  X_{t}\right)   &  =\lambda_{m}\left(  X_{0}\right)
+\mathrm{tr}\left(  \int\nolimits_{0}^{t}D\lambda_{m}\left(  X_{s-}\right)
dX_{s}\right)  +\frac{1}{2}\int\nolimits_{0}^{t}\sum\limits_{\left(
i,j\right)  \left(  k,h\right)  }A_{ij,kh}\frac{\partial^{2}}{\partial
z_{ij}z_{kh}}\lambda_{m}\left(  X_{s}\right)  ds\label{dyn}\\
&  +\int\nolimits_{\left(  0,t\right]  \times\mathbb{H}_{d}^{0}}\left[
\lambda_{m}\left(  X_{s-}+y\right)  -\lambda_{m}\left(  X_{s-}\right)
-\mathrm{tr}\left(  D\lambda_{m}\left(  X_{s-}\right)  y\right)  \right]
J_{X}\left(  ds,dy\right)  ,\nonumber
\end{align}
where a)
\[
D\lambda_{m}\left(  X_{s-}\right)  =\left(  \frac{\partial\Phi_{m}%
(s-)}{\partial x_{kh}}+\mathrm{i}\frac{\partial\Phi_{m}(s-)}{\partial y_{kh}%
}1_{\left\{  k\neq h\right\}  }\right)  _{1\leq k\leq h\leq d},
\]%
\begin{equation}
\frac{\partial\Phi_{m}(s-)}{\partial x_{kh}}=2\operatorname{Re}\left(
\overline{u}_{km}(s-)u_{hm}(s-)\right)  1_{\left\{  k\neq h\right\}
}+2\left\vert u_{km}(s-)\right\vert ^{2}1_{\left\{  k=h\right\}  ,} \label{dx}%
\end{equation}%
\begin{equation}
\frac{\partial\Phi_{m}(s-)}{\partial y_{kh}}=2\operatorname{Im}\left(
\overline{u}_{km}(s-)u_{hm}(s-)\right)  1_{\left\{  k\neq h\right\}  }.
\label{dy}%
\end{equation}
b)
\begin{equation}
\sum\limits_{\left(  i,j\right)  \left(  k,h\right)  }A_{ij,kh}\frac
{\partial^{2}}{\partial z_{ij}z_{kh}}\lambda_{m}\left(  X_{s}\right)
=2\sum\limits_{k=1}^{d}A_{kk,kk}\sum\limits_{j\neq m}\frac{\left\vert
u_{km}(s)u_{kj}(s)\right\vert ^{2}}{\lambda_{m}(s)-\lambda_{j}(s)}%
~~~~~~~~~~~~~~~~~~~~~~~~~~~ \label{d2}%
\end{equation}%
\[
~~~~~~~~~~~~~~~~~~~~~~+4\sum\limits_{1\leq k<h\leq d}A_{kh,kh}\sum
\limits_{j\neq m}\frac{\left\vert u_{km}(s)u_{hj}(s)\right\vert ^{2}%
+\left\vert u_{hm}(s)u_{kj}(s)\right\vert ^{2}}{\lambda_{m}(s)-\lambda_{j}%
(s)}\text{.}%
\]

\end{theorem}

\begin{proof}
By Lemma \ref{Der}, for each $m=1,2,...,d$ there exists $\Phi_{m}%
:\mathbb{R}^{d^{2}}\rightarrow\mathbb{R}$ which is $C^{\infty}$ in an open
subset $\widetilde{G}$ of $\mathbb{R}^{d^{2}}$ such that $\lambda_{m}%
(t)=\Phi_{m}\left(  X_{t}\right)  $ for each $t>0$, and the $\Phi_{m}$ satisfy
(\ref{fix})--(\ref{fiy2}).

Now, taking $f=\lambda_{m}$ in Proposition \ref{ito2}, where the open set $G$
is identified with the open subset $\widetilde{G}$, we have that $\lambda
_{m}\left(  X_{t}\right)  $ is a semimartingale where formula (\ref{iflp})
becomes (\ref{dyn}). The derivative matrix $D\lambda_{m}\left(  X_{s-}\right)
$ in (\ref{dyn}), see (\ref{DM}), is given by $D\lambda_{m}\left(
X_{s-}\right)  =\left(  \frac{\partial\Phi_{m}(s-)}{\partial x_{kh}%
}+\mathrm{i}\frac{\partial\Phi_{m}(s-)}{\partial y_{kh}}1_{\left\{  k\neq
h\right\}  }\right)  _{1\leq k\leq h\leq d},$ where the expressions of
$\frac{\partial\Phi_{m}(s-)}{\partial x_{kh}}$ and $\frac{\partial\Phi
_{m}(s-)}{\partial y_{kh}}$ in (\ref{fix}) and (\ref{fiy}) turn out to be in
(\ref{dx}) and (\ref{dy}).

The second derivative term in (\ref{dyn}) is given by $\sum\limits_{\left(
i,j\right)  \left(  k,h\right)  }A_{ij,kh}\frac{\partial^{2}}{\partial
z_{ij}z_{kh}}\lambda_{m}\left(  X_{s}\right)  =\sum\limits_{\left(
i,j\right)  \left(  k,h\right)  }A_{ij,kh}\frac{\partial^{2}}{\partial
z_{ij}z_{kh}}\Phi_{m}(s)$. It can be seen that $\frac{\partial^{2}}{\partial
x_{ii}x_{jj}}\Phi_{m}(s)=0$ for $i\neq j$, $\frac{\partial^{2}}{\partial
x_{ij}x_{kh}}\Phi_{m}(s)=0$ for $i<j,k<h$ with $ij\neq kh,\frac{\partial^{2}%
}{\partial x_{ij}y_{kh}}\Phi_{m}(s)=0$ for $i\leq j,k<h$, and $\frac
{\partial^{2}}{\partial y_{ij}y_{kh}}\Phi_{m}(s)=0$ for $i<j,k<h$ with $ij\neq
kh$.

Using the previous computations, the operator (\ref{OSD}) applied to $\Phi
_{m}(s)$ gives%

\begin{align}
\sum\limits_{\left(  i,j\right)  \left(  k,h\right)  }A_{ij,kh}\frac
{\partial^{2}}{\partial z_{ij}z_{kh}}\Phi_{m}(s)  &  =\sum\limits_{k=1}%
^{d}A_{kk,kk}\frac{\partial^{2}}{\partial x_{kk}^{2}}\Phi_{m}(X_{s}%
)\nonumber\\
&  +\sum\limits_{1\leq k<h\leq d}A_{kh,kh}\left(  \frac{\partial^{2}}{\partial
x_{kh}^{2}}+\frac{\partial^{2}}{\partial y_{kh}^{2}}\right)  \Phi_{m}%
(X_{s})\text{,} \label{ofi}%
\end{align}
and now from (\ref{fix2}) and (\ref{fiy2}) we get%
\begin{equation}
\sum\limits_{k=1}^{d}A_{kk,kk}\frac{\partial^{2}}{\partial x_{kk}^{2}}\Phi
_{m}(X_{s})=2\sum\limits_{k=1}^{d}A_{kk,kk}\sum\limits_{j\neq m}%
\frac{\left\vert \overline{u}_{km}u_{kj}\right\vert ^{2}}{\lambda_{m}%
-\lambda_{j}}\text{,} \label{ofi1}%
\end{equation}
and%
\begin{equation}
\sum\limits_{1\leq k<h\leq d}A_{kh,kh}\left(  \frac{\partial^{2}}{\partial
x_{kh}^{2}}+\frac{\partial^{2}}{\partial y_{kh}^{2}}\right)  \Phi_{m}%
(X_{s})~~~~~~~~~~~~~~~~~~~~~~~~~~~~~~~~~~~~~~~~~~~~~\ \ \ ~\ ~~~~~~~~~~~~~~
\label{ofi2}%
\end{equation}%
\begin{align*}
~~~~~~~~~~~~~~~~~~~~~~~~~~  &  =2\sum\limits_{1\leq k<h\leq d}A_{kh,kh}%
\sum\limits_{j\neq m}\left\{  \frac{\left\vert \overline{u}_{km}%
u_{hj}+\overline{u}_{hm}u_{kj}\right\vert ^{2}}{\lambda_{m}-\lambda_{j}}%
+\frac{\left\vert \overline{u}_{km}u_{hj}-\overline{u}_{hm}u_{kj}\right\vert
^{2}}{\lambda_{m}-\lambda_{j}}\right\} \\
&  =2\sum\limits_{1\leq k<h\leq d}A_{kh,kh}\sum\limits_{j\neq m}%
2\frac{\left\vert \overline{u}_{km}(s)u_{hj}(s)\right\vert ^{2}+\left\vert
\overline{u}_{hm}(s)u_{kj}(s)\right\vert ^{2}}{\lambda_{m}(s)-\lambda_{j}%
(s)}\text{.}%
\end{align*}
Finally, (\ref{d2}) follows from (\ref{ofi}), (\ref{ofi1}), (\ref{ofi2}).
\end{proof}

\begin{remark}
\label{rem}(a) Theorem \ref{Dynamics} is the analog of the Dyson--Brownian
motion (\ref{1}) for a Hermitian L\'{e}vy process. In this case, $X(t)$ has a
simple spectrum almost surely for each $t>0.$ The problem of noncolliding
eigenvalues for each $t>0$ almost surely will be considered elsewhere.

(b) However, we observe that the repulsion term between each pair of
eigenvalues $1/(\lambda_{i}(s)-\lambda_{j}(s))$, $i\neq j,$ appears in
(\ref{dyn}) and (\ref{d2}) when there is a Gaussian component (i.e.,
$\mathcal{A}\neq0$).

(c) The corresponding expression of Proposition \ref{ito3} can be obtained for
the dynamics of eigenvalues in terms of the local martingale and bounded
variation components of the eigenvalues. Let $\left\{  X_{t}:t\geq0\right\}  $
be a Hermitian L\'{e}vy process with triplet $\left(  \mathcal{A},\nu
,\Psi\right)  $ satisfying the same conditions as in Theorem \ref{Dynamics}.
For each $m=1,2,...,d$ it holds the decomposition $\lambda_{m}(X_{t}%
)=M_{t}^{m}+V_{t}^{m}$ where $M_{t}^{m}$ is a martingale and $V_{t}^{m}$ is a
process of bounded variation such that
\begin{align*}
M_{t}^{m}  &  =\lambda_{m}\left(  X_{0}\right)  +\mathrm{tr}\left(
\int\nolimits_{0}^{t}D\lambda_{m}\left(  X_{s-}\right)  B_{\mathcal{A}%
}(ds)\right) \\
&  ~~~~~~~~~~~~~~~~~~~~~~~~~~~~~~~~~~~~~~~~+\int\nolimits_{\left(  0,t\right]
\times\mathbb{H}_{d}^{0}}\left[  \lambda_{m}\left(  X_{s-}+y\right)
-\lambda_{m}\left(  X_{s-}\right)  \right]  \widetilde{J}_{X}\left(
ds,dy\right)
\end{align*}
and%
\[
V_{t}^{m}=\frac{1}{2}\int\nolimits_{0}^{t}\sum\limits_{\left(  i,j\right)
\left(  k,h\right)  }A_{ij,kh}\frac{\partial^{2}}{\partial z_{ij}z_{kh}%
}\lambda_{m}\left(  X_{s}\right)  ds+\mathrm{tr}\left(  \int\nolimits_{0}%
^{t}\Psi D\lambda_{m}\left(  X_{s}\right)  ds\right)  ~~~~~~~~~~~~~~~~~
\]%
\[
~~~~~~~~~~~~~~~~+\int\nolimits_{\left(  0,t\right]  \times\mathbb{H}_{d}^{0}%
}\left[  \lambda_{m}\left(  X_{s-}+y\right)  -\lambda_{m}\left(
X_{s-}\right)  -\mathrm{tr}\left(  D\lambda_{m}\left(  X_{s-}\right)
y\right)  1_{\{ \left\Vert y\right\Vert \leq1\}}\right]  ds\nu(dy).
\]

d) Hermitian L\'{e}vy processes with rank one jumps and with their
distributions invariant under unitary conjugations appear in the context of a
connection between classical and free one-dimensional infinitely divisible
distributions \cite{BG}, \cite{CD}, \cite{DRA}, \cite{DPARA}. Let us consider
an example with the triplet $\left(  \sigma^{2}\mathrm{I}_{d^{2}},\nu
,\psi\mathrm{I}_{d}\right)  $. From (c), for each $m=1,...,d,$ the eigenvalue
$\lambda_{m}$ has the decomposition $\lambda_{m}(t)=M_{t}^{m}+V_{t}^{m}$,
where
\begin{align*}
M_{t}^{m}  &  =\lambda_{m}\left(  X_{0}\right)  +\mathrm{tr}\left(
\int\nolimits_{0}^{t}D\lambda_{m}\left(  X_{s-}\right)  B_{\mathcal{A}%
}(ds)\right) \\
&  ~~~~~~~~~~~~~~~~~~~~~+\int\nolimits_{\left(  0,t\right]  \times
\mathbb{H}_{d}^{1}}\left[  \lambda_{m}\left(  X_{s-}+u\right)  -\lambda
_{m}\left(  X_{s-}\right)  \right]  \left(  J_{X}\left(  ds,du\right)
-ds\nu(du)\right)
\end{align*}
and%
\begin{align*}
V_{t}^{m}  &  =2\sigma^{2}\int\nolimits_{0}^{t}\sum\limits_{j\neq m}\frac
{1}{\lambda_{m}\left(  X_{s}\right)  -\lambda_{j}\left(  X_{s}\right)
}ds+\psi\mathrm{tr}\left(  \int\nolimits_{0}^{t}D\lambda_{m}\left(
X_{s}\right)  ds\right) \\
&  ~~~~+\int\nolimits_{\left(  0,t\right]  \times\mathbb{H}_{d}^{1}}\left[
\lambda_{m}\left(  X_{s-}+u\right)  -\lambda_{m}\left(  X_{s-}\right)
-\mathrm{tr}\left(  D\lambda_{m}\left(  X_{s-}\right)  u\right)  1_{\{
\left\Vert u\right\Vert \leq1\}}\right]  ds\nu(du).
\end{align*}
In particular, when the process is symmetric, the L\'{e}vy measure is
symmetric and $\psi=0,$ and hence%
\[
V_{t}^{m}=2\sigma^{2}\int\nolimits_{0}^{t}\sum\limits_{j\neq m}\frac
{1}{\lambda_{m}\left(  X_{s}\right)  -\lambda_{j}\left(  X_{s}\right)
}ds+\int\nolimits_{\left(  0,t\right]  \times\mathbb{H}_{d}^{1}}\left[
\lambda_{m}\left(  X_{s-}+u\right)  -\lambda_{m}\left(  X_{s-}\right)
\right]  ds\nu(du).
\]

\end{remark}

\noindent\textbf{Acknowledgement}. \emph{Most of this work was done while
Alfonso Rocha-Arteaga was visiting CIMAT in the summer of 2014 and Victor
P\'{e}rez-Abreu was visiting the Autonomous University of Sinaloa in the
winter of 2015. The authors would like to thank Michal Wojtylak for useful
comments on the use of Proposition }$2.2$ of \cite{RaWo11}\emph{, and Armando
Dom\'{\i}nguez, Jos\'{e}-Luis P\'{e}rez and the referee for valuable
comments.}

\end{document}